\documentclass[a4paper,10pt]{article}
\usepackage{QED}
\usepackage{amsmath}
\usepackage{amssymb}
\usepackage{a4}

\def\G{\Gamma}

\def\t{\tau}
\def\d{\delta}

\def\l{\lambda}

\def\ep{\epsilon}
\def\m{\mu}

\def\o{\omega}
\def\ep{\varepsilon}
\def\f{\rightarrow}
\def\tr{\triangleright}

\def\v{\vdash}

\def\ou{\vee}
\def\et{\wedge}
\def\<{\langle}
\def\>{\rangle}
\def\F{\displaystyle\frac}
\newtheorem{theorem}{Theorem}[section]
\newtheorem{lemma}{Lemma}[section]

\newtheorem{definition}{Definition}[section]
\newtheorem{notation}{Notation}[section]
\newcommand{\st}[1]{\diamond_ {#1} }
\newcommand{\ov}[1]{\overrightarrow{#1}}
\newcommand{\bo}[1]{\circ_{#1}}
\def\tst{{\cal T}'}
\def\ct{{\cal T}}
\def\trs{\blacktriangleright}
\begin{document}
\vspace*{0cm}

\begin{center}
{\Large\bf
A short proof of the Strong Normalization of Classical  Natural Deduction
with Disjunction}\\[1cm]
\end{center}

\begin{center}
{\bf Ren\'e David and Karim NOUR}\\
LAMA - Equipe de Logique\\
 Universit\'e de Chamb\'ery\\
73376 Le Bourget du Lac\\
e-mail: \{david,nour\}@univ-savoie.fr\\[1cm]
\end{center}

\begin{abstract}
We give a direct, purely arithmetical and elementary proof of the
strong normalization of the cut-elimination procedure for  full
(i.e. in presence of all the usual connectives) classical natural
deduction.
\end{abstract}

\section{Introduction}

This paper gives a direct proof of the strong normalization of the
cut-elimination procedure for  full propositional classical logic.
By full we mean that all the connectives ($\f$, $\et$ and $\ou$)
and $\perp$ (for the absurdity) are considered as primitive and
they have their intuitionistic meaning. As usual, the negation is
defined by $\neg A = A \f \perp$.

It is well known that, when the underlying logic is the classical one
(i.e. the absurdity rule is allowed) these connectives are redundant
since, for example, $\ou$ and $\et$ can be coded by using $\f$ and
$\perp$. From a logical point of view, considering the full logic is
thus somehow useless. However, from the computer science point of
view, considering the full logic is interesting because, by the
so-called Curry-Howard isomorphism, formulas can be seen as types for
functional programming languages and correct programs can be extracted
from proofs. For that reason various systems have been studied in the
last decades (see, for example,
\cite{BaBe,CoMu,Gir,Kri,Mur1,Mur2,Par1,ReSo}) both for intuitionistic
and classical logic. The connectives $\et$ and $\ou$ have a functional
counter-part ($\et$ corresponds to a product and $\ou$ to a
co-product, i.e. a {\em case of}) and it is thus useful to have them
as primitive.

Until very recently (see the introduction of \cite{deG2} for a
brief history), no proof of the strong normalization of the
cut-elimination procedure was known for full logic. In
\cite{deG2}, de Groote gives such a proof by using a CPS-style
transformation from full classical logic to intuitionistic logic
with $\f$ as the only connective, i.e. the simply typed
$\l$-calculus. A very elegant and direct proof of the strong
normalization of the full logic is given in \cite{JoMa} but only
the intuitionistic case is given.

We give here another proof of de Groote's result. This proof is
based on a proof of the strong normalization of the simply typed
$\lambda$-calculus due to the first author (see \cite{dav1})
which, itself, is a simplification of the one given by Matthes in
\cite{JoMa}. After this paper had been written we were told by
 Curien and some others that this kind of technique was
already present in van Daalen (see \cite{vda}) and  Levy (see
\cite{jjl}). The same idea is used in \cite{davnou} to give a
short proof of the strong normalization of the simply typed
$\l\m$-calculus of \cite{Par1}. Apart the fact that this proof is
direct (i.e. uses no translation into an other system whose strong
normalization is known) and corresponds to the intuition (the main
argument of the proof is an induction on the complexity of the
cut-formula) we believe that our technique is quite general and
may be used in other circumstances.  A crucial lemma of our proof
is used in \cite{Nour} to give a semantical proof of the strong
normalization. Finally \cite{dav2}  uses the same technique to
give an elementary proof of the strong normalization of a typed
$\l$-calculus with explicit substitutions which, from the logical
point of view, correspond to explicit cuts and weakenings.

\section{The typed system}

We code proofs by using a set of terms (denoted ${\cal T}$) which
extends the $\l\m$-terms of Parigot \cite{Par1} and is given by
the following grammar where $x,y,...$ are (intuitionistic)
variables and $a,b, ...$ are (classical) variables:
$$
{\cal T} ::= x \mid \l x {\cal T} \mid ({\cal T} \; {\cal E}) \mid
\< {\cal T} , {\cal T} \> \mid \o_1 {\cal T} \mid \o_2 {\cal T}
\mid \mu a {\cal T} \mid (a \; {\cal T})
$$
$$
{\cal E} ::=  {\cal T} \mid  \pi_1 \mid \pi_2 \mid [x.{\cal T} , y.{\cal T}]
$$


 The meaning of the new constructors is given by the typing rules of
figure 1 below where $\G$ is a context, i.e.  a set of
declarations of the form $x : A$ and $a : \neg A$ where $x$ is an
intuitionistic variable, $a$ is a classical variable  and $A$ is a
formula.

Note that, since we only are concerned with the logical point of
view, we should only consider typed terms, i.e. use a
$\l$-calculus \`a la Church. However, for the simplicity of
notation, the set of terms has been given in an untyped formalism
i.e. we use a $\l$-calculus \`a la Curry.

\begin{center}
$\F{}{\G , x : A \v x : A} \, ax$ \hspace{0.5cm} $\F{\G_1 \v M : A
\quad \G_2 \v N : B} {\G_1 , \G_2 \v \< M,N \> : A \et B } \,
\et_i$ \hspace{0.5cm} $\F{\G \v M : A_1 \et A_2} {\G \v  (M \; \pi_i):
A_i } \, \et_e$
\medskip

$\F{\G, x : A \v M : B} {\G \v \l x M : A \f B} \, \f_i$
\hspace{0.5cm}  $\F{\G_1 \v M : A \f B \quad \G_2 \v N : A} {\G_1,\G_2 \v (M \; N): B }\, \f_e$
\hspace{0.5cm} $\F{\G \v M : A_i} {\G \v \o_i M : A_1 \ou A_2} \, \ou_i$
\medskip

$\F{\G \v M : A_1 \ou A_2 \quad \G_1 , x_1 : A_1 \v N_1 : C \quad
\G_2 , x_2 : A_2 \v N_2 : C} {\G,\G_1,\G_2 \v (M\;[x_1.N_1,
x_2.N_2]) : C} \, \ou_e$
\medskip

$\F{\G , a : \neg A  \v M : \bot} {\G \v \mu a M : A }  \, abs_i$
\hspace{0.5cm} $\F{\G , a : \neg A  \v M : A} {\G \v ( a \; M) :
\bot } \, abs_e$
\medskip

Figure 1.
\end{center}

This coding is essentially the same as in \cite{and} and
\cite{deG2}. We have adopted the notations of \cite{and} which are
  also used by \cite{JoMa}: what is written $\pi_i M$ in \cite{deG2} is written
  $(M \; \pi_i)$ here and  $\d(M, x_1.N_1, x_2.N_2)$ in \cite{deG2} is written
  $(M \; [x_1.N_1, x_2.N_2])$ here. These notations have the advantage of
  making the permutative and classical reduction rules more
  uniform and thus simplifies the proofs.

The cut-elimination procedure corresponds to the  reduction rules given
below. There are three kinds of cuts.

{\em Logical cuts}: they appear when the introduction of a
 connective ($\f$, $\et$ and $\ou$) is immediately followed by its
 elimination.  The corresponding rules are:
\begin{itemize}
\item $(\l x M \; N) \tr M[x:=N]$
\item $(\< M_1,M_2 \> \; \pi_i) \tr M_i$
\item $(\o_i M\; [x_1. N_1 ,x_2. N_2]) \tr N_i[x_i:=M]$
\end{itemize}

{\em Permutative cuts }: they  appear when the elimination rule of
the disjunction is followed by the elimination rule of a
connective.  They are considered as cuts because a logical cut may
be hidden by the $\ou_e$ rule. Considering these cuts is necessary
to get the sub-formula property. The corresponding rule is:
\begin{itemize}
\item $(M \; [x_1.N_1 , x_2.N_2] \; \ep)
\tr (M \; [x_1.(N_1 \; \ep) , x_2.(N_2 \; \ep)])$
\end{itemize}

{\em Classical cuts }: they  appear when the classical rule  is
followed by the elimination rule of a connective. The
corresponding rule is:
\begin{itemize}
\item $(\mu a M \; \ep) \tr \mu a M[a:=^*  \ep]$ where $M[a:=^*
\ep]$ is obtained by replacing each sub-term of $M$ of the form
$(a \; N)$ by $(a \; (N \; \ep))$.
\end{itemize}

\begin{notation} Let $M$ be in ${\cal E}$.
\begin{enumerate}
\item  $M \tr M'$   means that $M$ reduces to $M'$
 by using one step of the reduction rules given
 above. As usual,  $\tr^+$  (resp. $\tr^*$) is the transitive
  (resp. reflexive and transitive) closure of $\tr$.
 \item $M$ is strongly normalizable (this is denoted by $M \in SN$) if
there is no infinite sequence of $\tr$ reductions.

\end{enumerate}
\end{notation}

\noindent {\em Remark} \hspace{.1cm} If $M=[y_1.N_1,y_2.N_2]$, $M
\tr M'$ means that $M'$ is either $[y_1.N'_1,y_2.N_2]$ or
$[y_1.N_1,y_2.N'_2]$ where $N_1 \tr N'_1$ or $N_2 \tr N'_2$. It is
thus clear that $M \in SN$ iff $N_1, N_2 \in SN$.

\medskip

The following result is straightforward.

\begin{lemma}[Subject reduction]
If $\G \v M : A$ and $M \tr^* N$ then $\G \v N : A$.
\end{lemma}

The goal of this paper is the proof of theorem \ref{SN} below.

\begin{theorem}\label{SN}
Every typed term is strongly normalizable.
\end{theorem}

The proof is an immediate corollary of theorem \ref{subst}: if
$M,N \in SN $, then $M[x:=N] \in SN$.

The proof of theorem \ref{subst}  uses a characterization of
strongly normalizable (theorem \ref{carSN}): a term is in $SN$ iff
its arguments and head reduct (see definition \ref{def}) are in
$SN$. This  theorem  needs another result (theorem \ref{app})
which is, intuitively, very clear but whose formal proof needs
some work.

The main difficulties  are the following:

- The first one is  minor: in the $\l$-calculus, each term has a
unique {\em head}, either a head variable or a head redex. Due to
the connective $\et$, this is no longer true here and a term may
have both a head variable and a head redex. This is treated by
showing that it is enough to consider only the simple terms (see
definition \ref{simp}).

- The second one is crucial and due to the presence of critical
pairs such as $(\m a  M  \; [y_1.N_1,y_2.N_2] \; \ep)$. We can
choose as head redex either the classical one or the permutative
one. If we choose the classical one, the proof of theorem
\ref{carSN} will be easy but the proof of theorem \ref{subst} does
not work because, in the rule $\vee_e$, the type of the main
hypothesis  has nothing to do with the type of its conclusion. We
thus have to choose, as head redex, the permutative one but then,
theorem \ref{carSN} needs the difficult theorem \ref{app}. For the
same reason (the rule $\vee_e$), the proof of theorem \ref{subst}
needs a rather complex induction: we use a 5-tuple of integers.
Note that E. Tahhan Bittar \cite{tahhan} has given a proof of the
strong normalization of the {\em sequent
  calculus} by using essentially the same 5-tuple of integers.

\medskip
\noindent {\em Remark} \hspace{.1cm} It is also for simplicity of
proofs that, in the totality of this section, we only consider
typed terms and thus, for example, that terms such as $(\l x
M\;[x_1.N_1,x_2.N_2])$ are not allowed because they, obviously,
cannot be typed since the type of $\l x M$ must be an implication.
Actually, theorems \ref{carSN} and \ref{app} would also be true
for untyped terms i.e. even if terms as $(\l x
M\;[x_1.N_1,x_2.N_2])$ were allowed and its proof will be
essentially the same since such a term is strongly normalizable
iff $M, N_1, N_2$ also are strongly normalizable.

\section{Characterization of strongly normalizable terms}

\begin{definition}\label{simp}
\begin{enumerate}
\item A term $M$ is simple if $M$  either is a variable or
  an application.
\item The set of contexts is given by the following grammar:

$$C:=*_i \mid  \l x C \mid \o_i C \mid \< C_1, C_2\> \mid \m a C$$

\item If $C$ is a context with holes $*_1,...,*_n$ and $M_1,...,M_n$
  are terms, $C[M_1,...,M_n]$ is the term obtained by
  replacing each $*_i$ by $M_i$.
\end{enumerate}
\end{definition}

\begin{lemma}\label{simple} Each term $M$ can be uniquely written as $C[M_1,...,M_n]$
  where $C$ is a context  and $M_1,...,M_n$ are simple terms.
\end{lemma}

\begin{proof} By induction on $M$.
\end{proof}

\begin{lemma} \label{redSN}Let $C$ be a context and
  $M_1,...,M_n$ be terms. Then $C[M_1,...,M_n] \in SN$ iff
$M_1,...,M_n \in SN$.
\end{lemma}

\begin{proof} By induction on $C$.
\end{proof}

\begin{definition}[and notation]
A (possibly empty) sequence $\ov{N}=N_1,...,N_n$  of elements of
${\cal E}$
 is nice if each $N_i \in
{\cal{T}} \cup \{\pi_1, \pi_2\}$ except possibly
  for $i=n$. If $M$ is a
 term,  $(M \; N_1...N_n)$
will be denoted as  $(M
  \;\ov{N})$.

\end{definition}

\begin{lemma}[and definition]\label{def}
Let $M$ be a simple term.

\begin{enumerate}
  \item Then $M$ can be uniquely written  as
one of the cases of the figure below where $\ov{T}=T_1, ..., T_n$
is a nice sequence and, in case (4) and (5), $\ep \ \ov{T}$ is
also nice, i.e. if $\ep=[y_1.N_1,y_2.N_2]$ then $\ov{T}$ is empty.
  \item The set of arguments of $M$ (denoted as $arg(M)$) and the head
of $M$ (denoted as $hd(M)$), either a redex  or a variable, are
defined by the figure below.

  \item The head reduct of $M$ (denoted as $hred(M)$) is the term obtained by reducing, if any,
   the head redex of $M$.
\end{enumerate}

\begin{center}
\medskip
\hspace{-1 cm}
\begin{tabular}{|r|l|c|c|}
\hline  & $M$  &  $hd(M)$  &  $arg(M)$\\
\hline  0& $(x \; \ov{T})$ or $(a \; T)$ &  $x$ or $a$ & $\{T_1, ..., T_n\}$  or $T$\\
\hline 1 & $(\lambda x N \; O \; \ov{T})$  &   $(\lambda x N
\, O)$ &  $\{O\}$\\
\hline 2 & $(\<N_1,N_2\> \; \pi_i\; \ov{T})$  &
$(\<N_1,N_2\> \; \pi_i)$ &  $\{N_1, N_2\}$  \\
\hline 3 & $(\o_iN \; [x_1.O_1,x_2.O_2])$  &   $M$ &  $\{N,O_1, O_2\}$\\
\hline 4 & $(\mu a N \; \ep \; \ov{T})$  &   $(\mu a N
\, \ep)$ &  $\{\ep\}$ \\
\hline 5 & $(N \; [x_1.O_1,x_2.O_2] \; \ep \; \ov{T})$ &
$(N \; [x_1.O_1,x_2.O_2] \; \ep)$   & $\emptyset$   \\
\hline
\end{tabular}
\end{center}

\end{lemma}

\medskip

 \begin{proof}
 Since $M$ is simple, and for trivial typing reasons, it looks like either (a) $(x \; \ov{S})$ or $(a \; S)$
 or (b)
 $(\lambda x N \; O \; \ov{S})$
or (c) $(\<N_1,N_2\> \; \pi_i \; \ov{S})$ or (d) $(\o_iN \;
[x_1.O_1,x_2.O_2] \; \ov{S})$ or (e) $(\mu a N \; \ep \; \ov{S})$.
If $\ov{S}$ is empty the result is clear.

Otherwise, assume first $\ov{S}$ is nice. The cases (a), (b) and
(c) are clear. Case (d) gives (5). Case (e) gives (5) if
$\ep=[y_1.N_1,y_2.N_2]$ or (4) otherwise.

Assume finally $\ov{S}$ is not nice. Then $\ov{S}$ can be written
as $\ov{S_1} [y_1.N_1,y_2.N_2] \ov{S_2}$ where $\ov{S_2}$ is nice
and non empty. It is then easy to see that, in all cases, this
gives (5) where $\ov{S_2}= \ep \ov{T}$.

For uniqueness, check easily (by looking wether $\ov{T}$ has an
$[y_1.N_1,y_2.N_2]$ or not) that if $M$ is in case 0 to 4 it
cannot also be in case 5.
\end{proof}

\begin{theorem}\label{carSN} Let $M$ be a simple term. If $M$ has an head redex, then $M \in SN$
  iff $arg(M) \subset SN$ and $hred(M) \in SN$. Otherwise, $M \in SN$
  iff $arg(M) \subset SN$.
\end{theorem}

\begin{proof}
 The case  of an head variable is trivial.
Case 1 of the figure of lemma \ref{def} is done as follows. Since
$hred(M) \in SN$, $N$ and $\ov{T}$ are in $SN$. Thus and since
$\ov{T}$ is nice, an infinite reduction of $M$ must look like: $M
\tr^* (\lambda x N_1 \; O_1 \; \ov{T_1}) \tr (N_1[x:=O_1] \;
\ov{T_1}) \tr ...$. The contradiction comes from the fact (see
lemma \ref{l7} below) that $hred(M) \tr^* (N_1[x:=O_1] \;
\ov{T_1})$. Cases 2, 3, 4 are similar.

Case 5 is theorem \ref{app} below.
\end{proof}

\begin{lemma}\label{l7}
Let $M, N \in \ct$. Assume $M \tr M'$ and  $N \tr N'$. Let
$\sigma$ (resp. $\sigma'$) be either $[x:=N]$ or $[a:=^*N]$ (resp.
$[x:=N']$ or $[a:=^*N']$).  Then $M[\sigma] \tr M'[\sigma]$ and
$M[\sigma] \tr^* M[\sigma']$.
\end{lemma}
\begin{proof}
Straightforward.
\end{proof}

\begin{theorem}\label{app}
Assume the sequence $\ep \; \ov{V}$ is nice and $S_2=(N \; [x_1.
\; (N_1 \; \ep), x_2. \; (N_2 \; \ep)]$ $ \; \ov{V}) \in SN$. Then
$S_1=(M \; [x_1. \; N_1, x_2. \; N_2] \; \ep \; \ov{V}) \in SN$.
\end{theorem}

\begin{proof}
See section 5.
\end{proof}

\section{Proof of theorem \ref{SN}}

By induction on $M$.
 The cases $x$, $\l xN$, $\langle N, O\rangle$, $\omega
_{i}N$, $(a \; N)$ and $\mu aN$ are immediate. The last case is
 $M=(N \; \ep)=(x \; \ep)[x:=N]$ where $x$ is a fresh
variable and the result follows from the induction hypothesis and
theorem \ref{subst} below. \hfill $\square$

\begin{definition}
Let $M$ be a term. Then, $cxty(M)$ is the number of symbols
occurring in $M$ and, if $M \in SN$,
 $\eta(M)$ is the length of the longest reduction of $M$.
\end{definition}

In lemma \ref{utile} and theorem \ref{subst} below,  $\sigma$
denotes a substitution of the form
  $[x_i:=N_i \; / \; i=1...n]$, i.e. we substitute only intuitionistic variables.

\begin{lemma}\label{utile}
 Let $M$ be a simple term with an head redex and $\sigma$ be a substitution. Then,
  $hd(M[\sigma]) = hd(M)[\sigma]$, $arg(M[\sigma]) =
  arg(M)[\sigma]$ and $hred(M[\sigma]) = hred(M)[\sigma]$.
\end{lemma}

\begin{proof} Immediate.
\end{proof}

\begin{theorem} \label{subst}
Let $M \in SN$ be a term and $\sigma$ be a substitution. Assume
that the substituted variables all have the same type and, for all
$x$, $\sigma(x) \in SN$. Then $M[\sigma] \in SN$.
\end{theorem}

\begin{proof}

The proof is by induction
  on $(lgt(\sigma)),\eta (M),cxty(M),\eta
 (\sigma),cxty(\sigma))$ where $lgt(\sigma)$
is the number of connectives in the type of the substituted
variables and $\eta(\sigma)$ (resp. $cxty(\sigma)$) is the sum of the
 $\eta(N)$ (resp. $cxty(N)$) for the
$N$ that are actually substituted, i.e. for example if
$\sigma=[x:=N]$ and $x$ occurs $n$ times in $M$, then
$\eta(\sigma)=n.\eta(N)$ and $cxty(\sigma)=n.cxty(N)$. The
induction hypothesis will be abbreviated as {\it IH}.

By the {\it IH} and lemmas \ref{redSN} and \ref{simple} we may
assume that $M$ is simple.
Consider then the various cases of lemma \ref{def}.

\begin{itemize}
\item If $M$ has an head redex:   by lemma \ref{utile} and the {\it IH},
  $arg(M[\sigma]) \subseteq SN$ since for each $N \in arg(M)$,
  $cxty(N) < cxty(M)$. By lemma \ref{utile},  $hred(M[\sigma])
  = hred(M)[\sigma]$ and thus, since $\eta(hred(M)) < \eta(M)$,   $hred(M[\sigma]) \in SN$
  follows from the {\it IH}.
\item  Otherwise, if the head variable is a classical variable or an intuitionistic  variable
not in the domain
of $\sigma$, the result is trivial.
\item  Otherwise, i.e $M= (x \; \ov{T})$
\begin{itemize}
\item If $hd(M[\sigma])$ is a variable, the result is trivial.
\item If $hd(M[\sigma]) = hd(\sigma(x))$:
  let $M' = z
  \; \ov{T}$ where $z$ is a fresh variable and $\sigma'$ be the substitution defined as follows
  $\sigma'(z)= hred(\sigma(x))$ and, for  the variables $y$ occurring in $\ov{T}$, $\sigma'(y)=\sigma(y)$.
   Then,
  $hred(M[\sigma]) = M'[\sigma']$ and thus, by the {\it IH}, $hred(M[\sigma]) \in SN $
  since $\eta(\sigma') < \eta(\sigma)$.
\item  Otherwise, the head redex has been created by the
substitution. The various cases are:
\end{itemize}
\end{itemize}
\begin{enumerate}
\item \label{1}  $M = (x \; O \ov{S})$ and $\sigma(x) = \l y
N$.  By the {\it IH},  $arg(M[\sigma]) \subseteq SN$ and thus, by
theorem \ref{carSN}, we have to show that $P=(N[y:=O[\sigma]]
\;\ov{S[\sigma])} \in SN$.  By the {\it IH},  $(z \;
\ov{S[\sigma]}) \in SN$ and since $lgt(O[\sigma]) < lgt(\l yN)$,
$N[y:=O[\sigma]] \in SN$ . Thus $P = (z \; \ov{S[\sigma]})[z
:=N[y:=O[\sigma]]] \in SN$ since $lgt(N[y:=O[\sigma]])<lgt(\l
yN)$.

\item  $M = (x \; \pi_i \ov{S})$ and $\sigma(x) =
  \<N_1,N_2\>$ or $M = (x \; [x_1.M_1,x_2.M_2])$ and $\sigma(x) = \o_iN$.  The proof is similar.

\item \label{2} $M = (x \; [x_1.M_1,x_2.M_2])$ and $\sigma(x) = \mu
  a N$. By the {\it IH},  $arg(M[\sigma])
  \subseteq SN$ and thus (by theorem \ref{carSN}) we have to show
   $P=\mu a N[a:=^*  [x_1.P_1,x_2.P_2]] \in SN$ where, for $i=1,2$,
  $P_i=M_i[\sigma]$. Since $cxty(M_i) < cxty(M)$, the fact that
  $P_i \in SN$ follows from the {\it IH}. The result is thus a particular case of the claim
  below.

\emph{Claim} \hspace{0.1cm} Let $P_{1},P_{2}, T\in SN$ and
$a_{1},...,a_{n}$ be
 variables of type $\lnot (A\ou B)$.  Let
$T[\t]$ denotes $T[a_{i}:=^* [P] \; / \; i=1...n]$ where $[P]$ is
an abbreviation for  $[x_{1}.P_{1},x_{2}.P_{2}]$. Then $T[\t]\in
SN$.

\emph{Proof} \hspace{0.1cm} By induction on $(\eta (T),cxty(T))$.
We
  may assume that $T$ is simple.  Consider the various cases of
  lemma \ref{def}.

\begin{itemize}
\item If $T$ has an head redex, the result follows immediately from {\it IH} and
lemma \ref{utile}.
\item Otherwise and if the head variable of $T$ is not in $\t$, the result is trivial.

\item Otherwise and because of the type of the
  $a_i$, $T = (a \; V)$ where $V \in {\cal T}$. It is thus enough to prove
  that $(V[\t] \; [P] )\in SN$ and, for that, it is enough to show
  that its head reduct $Q \in SN$. The various cases are:
\begin{itemize}

\item $V=\omega _{i}W$ and $Q=P_{i}[x_{i}:=W[\t]]$. By the \textit{IH},
$W[\t]\in SN$ since $cxty(W)<cxty (T)$ and thus, since
$lgt(W)<lgt(N)$, $Q \in SN$ follows from  the main \textit{IH}
(recall we are ``inside'' the proof of theorem \ref{subst},
$type(W)=A$ or $type(W)=B$ and $type(N)=A\vee B$).

\item $V=\mu bW$ and $Q=\mu bW[\t][b:=^*[P]]=\mu
  b \; W[\t^{\prime }]$ where $\t^{\prime}=\t \cup \lbrack b:=^*[P]]$. Since $cxty(W)<cxty(T)$, the result follows from
  the \textit{IH}.

\item  $V=(W\,\ep)$ and $\ep$ is not in the form
$[x_1.W_1,x_2.W_2]$. Then, the head redex of $(V[\t] \; [P])$ must
come from $V$ and  $Q=(V'[\t] \; [P])$ for some $V'$ such that $V
\tr V'$. Let $T'=(a \;\,V')$. Since $\eta (T')<\eta (T)$,
$T'[\t]\in SN$. But $T'[\t] \tr Q$ and thus $Q\in SN$.

\item $V=(W \; [x_1.W_1,x_2.W_2])$ and $Q=(W[\t] \; [x_1.(W_1[\t] \;
  [P]),x_2.(W_2[\t] \; [P])])$.  Let $T_{j}=(a\;W_{j})$. Since
  $cxty(T_{j})<cxty(T)$, by the \textit{IH}, $T_{j}[\t]\in SN$ and
  thus $(W_{j}[\t] \; [P])\in SN$ since $T_{j}[\t] \tr (W_{j}[\t] \;
  [P])$. By the \textit{IH}, since $cxty(W)<cxty(T)$, $W[\t]\in
  SN$. By theorem \ref{carSN}, it is thus enough to show that $Q'=hred(Q) \in SN$.

  If $hd(Q)$  comes from $W$, the result follows
  from the \textit{IH}.  Otherwise, the various cases are:
\begin{itemize}
\item $W=\omega _{i}W^{\prime }$ and  $Q' = (W_{i}[\t] \; [P])[x_{i}:=W'[\t]]$.
 Let $T'=(a \,W_{i}[x_{i}:=W'])$. Then $T=(a \;(\o_{i}W'\;
  [x_1.W_1,x_2.W_2]))\tr T'$. By the \textit{IH}, $T'[\t]\in SN$ and
  the result follows from the fact that $T'[\t]\rightarrow (W_{i} \;
  [P])[x_{i}:=W'][\t]=Q'$.
\item  If $W=\mu bW^{\prime }$ or $W=(W' \; [x_1.W'_1,x_2.W'_2])$: the proof is
similar.

\end{itemize}
\end{itemize}
\end{itemize}
\item $M = (x \; \ep \; \ov{T})$, $\ep \neq [x_1.M_1,x_2.M_2]$ and $\sigma(x) = \mu a N$. We prove exactly as
in case \ref{2} that $(\mu a N \; \ep [\sigma]) \in SN$. To prove
that $M[\sigma] \in SN$, it is enough to use the same trick as in
case 1: $M[\sigma]= (z \; \ov{T}[\sigma])[z:=(\mu a N \; \ep
[\sigma])]$ where $z$ is a fresh variable and the \textit{IH}
gives the result since $lgt(z) < lgt(x)$.

\item \label{4} $M = (x \; [x_1.M_1,x_2.M_2])$ and $\sigma(x) = (N_{3} \;
[y_1.N_1,y_2.N_2])$. By theorem \ref{carSN}, it is enough to show
$P= (N_3 \; [y_1.(N_1 \; [P]),y_2.(N_2 \;[P])]) \in SN$ where, for
$i=1,2$, $P_i=M_i[\sigma]$ and $(N_i \; [P])$ is a notation for
$(N_i \; [x_1.P_1,x_2.P_2])$. Let $M' = (z \; [x_1.M_1,x_2.M_2])$
where $z$ be a fresh variable. For $i=1,2$, let $\sigma_i = \sigma
\cup {[z:=N_i]}$. By the {\it IH}, $M'[\sigma_i] \in SN$ since
$\eta(\sigma') \leq \eta(\sigma)$ and $cxty(\sigma') <
cxty(\sigma)$. Then $(N_i \; [P]) \in SN$ since $M'[\sigma_i] \tr
(N_i \; [P])$. By theorem \ref{carSN}, it is thus enough to show
that $Q=hred(P) \in SN$.

If $hd(P)$ comes from $W$, the result follows from the
\textit{IH}. Otherwise, the various cases are:

\begin{itemize}
\item $N_3=\omega_i N'_3$ and
  $Q=(N_{i}[x_{i}:=N'_3] \; [P])$. Let $M' = (z \;
  [x_1.M_1,x_2.M_2])$ where $z$ is a fresh variable and $\sigma' = \sigma
  \cup \{[z := N_{i}[x_{i}:=N'_3]\}$. Then $Q=M'[\sigma'] \in SN$
  since $\eta(\sigma') \leq \eta(\sigma)$ and $cxty(\sigma') < cxty(\sigma)$.

\item  $N_3 = \mu a N'_3$ or $N_3=(Q_3 \; [y_1.Q_{1},y_2.Q_2])$. The proof is
similar.
\end{itemize}

\item If $M = (x \; \ep \; \ov{T})$, $\ep \neq [x_1.M_1,x_2.M_2]$ and $\sigma(x) = (N_3 \;
[x_1.N_1,x_2.N_2])$. We prove exactly as in case \ref{4} that
$(N_3 \; [x_1.N_1,x_2.N_2] \;  \ep [\sigma]) \in SN$. To prove
that $M[\sigma] \in SN$, it is enough to use the same trick as in
case 1: $M[\sigma]= (z \; \ov{T}[\sigma])[z:=(N_3 \;
[x_1.N_1,x_2.N_2] \; \ep [\sigma])]$ where $z$ is a fresh variable
and the \textit{IH} gives the result since $lgt(z) < lgt(x)$.

\end{enumerate}
\end{proof}

\section{Proof of theorem \ref{app}}\label{App}

The idea of the proof is the following: we show that an infinite
reduction of $S_1$ can be translated into an infinite reduction of
$S_2$. These reductions are  the same except that, in $S_1$, $\ep$
can be far away from the $N_i$. We mark $\ep$ and the $N_i$ to
keep their trace. This gives the  set of marked terms $\tst$ of
definition \ref{mark}. The correct terms of definition \ref{cor}
intuitively are the marked terms for which each marked $N_i$ {\em
knows} who is the corresponding marked $\ep$. Concretely, being
correct is a sufficient condition to ensure that a reduction in
the marked $S_1$ can be translated to the corresponding $S_2$.

The main difficulty of the proof consists in  writing precise
definitions. The proofs of the lemmas consist in easy but tedious
verifications.

\noindent {\em Important remark.} The proof is  uniform in the
sequence $\ep \; \ov{V}$. In definition \ref{mark} below, we
implicitly assume the following: if we are proving theorem
\ref{app} for $\ep \in {\cal T}$ (resp. $\ep=\pi_i$, $\ep=[y_1. \;
M_1, y_2. \; M_2]$) then, in the sub-terms of the form
$\bo{\ep'}$, we necessarily have $\ep' \in {\cal T}$ (resp.
$\ep'=\pi_i$, $\ep'=[y_1. \; Q_1, y_2. \; Q_2]$). Note that we
could also assume that $\ep'$ is a reduct of $\ep$  but this does
not really matter for the proof. However, in the case $\ep=[y_1.
\; P_1, y_2. \; P_2]$, since the sequence $\ep \; \ov{V}$ is nice
$\ov{V}$ is empty and this must appear in the proof. We will do
the proof only for $\ep \in {\cal T}$ or $\ep=\pi_i$. The proof
for the case $\ep=[y_1. \; P_1, y_2. \; P_2]$ is essentially the
same: we just have to add  an third condition in definition
\ref{cor} and check in the lemmas that this condition is
preserved. This new condition is given in the final remark of this
section.

\begin{definition}\label{mark}
\begin{enumerate}
  \item Let $\tst$ be the set of terms obtained from ${\cal T}$ by
adding  new constructors: $\st{N}$ and $\bo{\ep}$ where $N \in
{\cal T}$ and $\ep \in {\cal E}$ are closed.
 \item The reduction rules for $\tst$ are the ones of $\ct$ plus
  the following:
\begin{itemize}
  \item If $ N \tr N'$ then $\st{N} \tr \st{N'}$ and
  $\bo{N} \tr \bo{N'}$.
  \item $(\st{N} \; \bo{\ep}) \tr (N \; \ep)$.
\end{itemize}
\item Let $\trs$ be the congruence defined by the following reduction
rules:
\begin{itemize}
  \item $(M \; [x_1.N_1, x_2.N_2] \; \bo{\ep}) \trs (M \; [x_1. (N_1 \;
\bo{\ep}),
  x_2. (N_2 \; \bo{\ep})])$
  \item $(\mu a M \; \bo{\ep}) \trs \mu a M[a:=^* \bo{\ep}]$
\end{itemize}
\end{enumerate}
\end{definition}

\noindent {\em Comments}

An element  of $\tst$ is a term in ${\cal T}$  where some
sub-terms have been replaced by  terms as $\st{N}$ or $\bo{\ep}$
where  $N \in {\cal T}$ and $\ep \in {\cal E}$ and, in particular,
have no sub-terms as  $\st{N'}$ or $\bo{\ep'}$. It is assumed, in
the definition, that the $N$ and $\ep$ occurring in $\st{N}$ or
$\bo{\ep}$ are closed. In fact, they are allowed to have free
variables (both intuitionistic and classical) but it is assumed
that these variables will never be captured and thus act as
constants.

\begin{definition}\label{acc}
Let $M \in \tst$.
\begin{enumerate}
  \item $M$ is acceptable iff $M=\st{N}$ or $M=\m a  M_1$ and, for each
sub-term of $M$ of the form $(a \; N)$, $N$ is acceptable or $M=(N
\; [x_1. N_1, x_2.N_2])$ and $N_1, N_2$ are acceptable.
 \item If $M$ is acceptable, the set $st(M)$ of terms is defined by:
$st(\st{N})=\{\st{N}\}$, $st(\m a  M_1)=\cup \{st(S)\; / \; (a \;
S)$  sub-term of $M_1 \}$ and $st((N \; [x_1. N_1, x_2.N_2]))=
st(N_1) \cup st(N_2)$.
\end{enumerate}
\end{definition}

\begin{lemma}\label{app3}
Let $M \in \tst$ be an acceptable term.
\begin{enumerate}
\item  If $\sigma$ is a substitution either of the form $[x:=N]$ or
$[a:=^*N]$, then $M[\sigma]$ is acceptable
  and $st(M[\sigma]) = st(M)$.
\item If $M \tr M'$, then $M'$ is acceptable and $st(M') \subseteq
st(M)$.
\end{enumerate}

\end{lemma}
\begin{proof} By induction on $M$. (1) trivial. For (2) use (1).
\end{proof}

\begin{definition}\label{cor}
A term $M \in \tst$ is correct if the following conditions hold.
\begin{enumerate}
  \item Each occurrence of a term of the
  form $\bo{\ep}$ appears as $(U \; \bo{\ep})$ for some acceptable term
$U$.
  \item  For each sub-term  of $M$ of the form $\st{N}$ there is
  a sub-term (necessarily  unique) of the form
  $(U \; \bo{\ep})$  such that $\st{N}$ belongs to
  $st(U)$. The corresponding $\ep$ is denoted as $eps(N)$
\end{enumerate}
\end{definition}

\noindent {\bf Examples}
\begin{itemize}
 \item Assume $M,N,O,P, \ep$ are closed terms.
Then $A=(M \;[x_1. \st{N}, x_2.  \st{O}]\; \bo{\ep} \; P)$ is
correct.

\item Assume $M,N,O,P,Q,R,S,\ep_1,\ep_2$  are closed terms.
Then $B=$

$(M \;[x_1.  (N \; [y_1. \; \st{O}, y_2. \; \mu a P] \;
\bo{\ep_1}),x_2. (\m b (b \; \mu c (c \; (Q \; [z_1.  \mu d R ,
z_2. \st{S}])))$ $ \; \bo{\ep_2})])$ is correct.
\end{itemize}

\begin{lemma}\label{T1}
If $M$ is correct and $M \tr M'$, then $M'$ is correct.
\end{lemma}

\begin{proof}  Let
$(U \; \bo{\ep})$ be a sub-term of $M$. A reduction can be, either
in $\ep$ or in $U$ or between $U$ and $\ep$ or, finally, above $(U
\; \bo{\ep})$. Since $U$ is acceptable and by using lemma
\ref{app3} it is easy to check that, in each case the conditions
of correctness are still satisfied.
\end{proof}

\begin{lemma}\label{car*}
Let $M$ be a correct term.
\begin{enumerate}
\item $M$ has no sub-term of the form $(O \,\, \st{N})$.
\item If $(\st{N} \,\, O)$ is a  sub-term of $M$, then  $O = \bo{\ep}$ for some $\ep$.
\end{enumerate}
\end{lemma}

\begin{proof} Otherwise, let $(U
\; \bo{\ep})$ be the sub-term such that $\st{N} \in st(U)$. The
result follows easily from the fact that $U$ is acceptable.
\end{proof}

\begin{definition}
Let $M \in \tst$.
\begin{enumerate}
  \item $T_1(M)$ is the term obtained by replacing $\st{N}$ by $N$ and
$\bo{\ep}$ by $\ep$.
  \item If $M$ is a sub-term of a correct term,  $T_2(M)$ is the term obtained by replacing
each occurrence of
 $(U \; \bo{\ep})$ by $U'$ where $U'$ is obtained from $U$ by replacing
 each occurrence of $\st{N}$ such that $\ep=eps(N)$ is a sub-term of $M$ by $(N \; \ep)$.
\end{enumerate}
\end{definition}
\medskip

\noindent {\bf Comments and examples}
\begin{enumerate}
\item If $M$ itself is correct, $T_2(M) \in \ct$. Otherwise, some $\st{N}$ that are related
to a $\bo{\ep}$ outside $M$ are not replaced. We need this more
general definition for the proof of lemma \ref{app1}.
\item If  $M$ is correct, $T_1(M)
  \tr^* T_2(M)$. More precisely $T_2(M)=T_1(M')$ where $M'$ is the normal form of $M$
  for the rules $\trs$.
 Since we will not use this result, we do not prove
  it.
\item Let $A,B$ be the terms of the previous
example. Then

$T_1(A)=(M \;[x_1. N, x_2.  O]\; \ep \; P)$ and $T_2(A)=(M \;[x_1.
(N \; \ep), x_2. (O \; \ep)] \; P)$.

$T_1(B)=(M \;[x_1.(N \; [y_1. \; O, y_2. \; \mu a P] \;
\ep_1),x_2.(\m b (b \; \mu c (c \; (Q \; [z_1.\mu d R, z_2.S])))
\; \ep_2)])$ and $T_2(B)=(M \;[x_1.  (N \; [y_1. \; (O \; \ep_1),
y_2. \;
      \mu a P]),x_2. \m b (b \; \mu c (c \; (Q \; [z_1.\mu d R ,$

      $z_2.(S \; \ep_2)])))])$.
\end{enumerate}

\begin{lemma}\label{app1}

Let $M \in \tst$ be correct. If $T_1(M) \tr N$,  there is a
correct term $M'$ such that $M \tr^+ M'$ and $T_1(M')=N$.
\end{lemma}

\begin{proof}
  Let $R$ be the redex that has been reduced. By
lemma \ref{car*}, the only cases to consider are:
\begin{itemize}
\item There is a redex $S$ in $M$  such that $R= T_1(S)$. The result follows then from
 lemma \ref{cor}.
 \item There is a sub-term of $M$ of the form  $\st{N}$ or $\bo{\ep}$,
 such that $R$ is a sub-term of $N$ or $\ep$. The result is then
 trivial.
\item Finally, $R=(T_1(U) \; \ep)$ where $(U \; \bo{\ep})$ is a sub-term of $M$, the result follows from
the fact that $U$ is acceptable. \qed
\end{itemize}
\end{proof}

\begin{lemma}\label{app4}
Let $P=(M \; O)$ be a sub-term of a correct term.   Assume $O \neq
\bo{\ep}$ and $P \tr P'$ by reducing the redex $P$. Then, $T_2(P)
\tr^* T_2(P')$.
\end{lemma}
\begin{proof}
It is, for example, straightforward to check that, if $M=\l x
M_1$, then $T_2(P)=(\l x T_2(M_1) \; T_2(O))$ and
$T_2(P')=T_2(M_1) [x:=T_2(O)]$. The other cases are similar.
\end{proof}

\begin{lemma}\label{app5}
Assume $M=(N \; \bo{\ep})$ is a sub-term of a correct term. Then,
\begin{itemize}
  \item $T_2(M)=T_2((T_2(N) \; \bo{\ep}))$.
  \item If $N$ has no sub-terms of the form $\bo{\ep'}$ and $N \tr
  N'$ then, $T_2(M) \tr^* T_2((N' \; \bo{\ep}))$.

\end{itemize}

\end{lemma}
\begin{proof}
Straightforward.
\end{proof}

\begin{lemma}\label{app1'}
Let $M $ is a sub-term of a correct term. If $M \tr M'$, then
$T_2(M) \tr^* T_2(M')$. Moreover, if
  $T_2(M)=T_2(M')$, then $M \trs M'$.
\end{lemma}

\begin{proof} By induction on $(nb(M), cxty(M))$ where $nb(M)$ is the number
 of sub-terms of the form $\bo{\ep}$ in $M$. The only non trivial  case is $M
  = (N \, O)$.
\begin{itemize}
\item Assume $O \neq \bo{\ep}$. If $M' =
  (N' \, O)$ where $N \tr N'$ or $M' = (N \, O')$ where $O \tr O'$,
  the result is trivial. Otherwise, $M$ itself is the reduced
  redex and the result comes from lemma \ref{app4}.
\item Assume $O = \bo{\ep}$. If $M$ itself is the reduced redex then $T_2(M)=T_2(M')$ and
$M \trs M'$. If $M' = (N \,
  \bo{\ep'})$ where $\ep \tr \ep'$, the result is trivial. Otherwise, i.e. $M' =
  (N' \, \bo{\ep})$ where $N \tr N'$. If $nb(N)=0$, the result follows from lemma \ref{app5}(2).
  Otherwise, by the induction
  hypothesis, $T_2(N) \tr^* T_2(N')$ and the result comes from lemma \ref{app5}.  \qed
\end{itemize}
\end{proof}

\begin{lemma}\label {app2}
Let $M \in \tst$ be a correct term. Then $M$ is strongly
normalizable for the $\trs$ reduction.
\end{lemma}

\begin{proof}
If $M$ is correct, let $lg(M)$ be the sum of the length of the
path (i.e. the number of nodes in the tree representing $M$)
relating the $\st{N}$ to the corresponding $\bo{\ep}$. It is easy
to see that, if $M \trs M'$, then $lg(M') < lg(M)$.
\end{proof}

\medskip

\noindent {\bf Proof of theorem \ref{app}}

Assume $S_2 \in SN$ and $S_1 \not\in SN$.  Let $(U_i)$ be a
sequence of terms such that $U_0=S_1$ and, for each $i$, $U_i \tr
U_{i+1}$. Let $M=(N \; [x_1. \;\st{N_1}, x_2. \; \st{N_2}]\;
\bo{\ep} \; \ov{V})$. By using lemma \ref{app1}, we get a sequence
of correct terms $M_i$ such that, for each $i$, $M_i \tr^+
M_{i+1}$ and $T_1(M_i)=U_i$. By lemma  \ref{app1'}, $T_2(M_i)
\tr^* T_2(M_{i+1})$.  Since $S_2=T_2(M) \in SN$, there is an $i_0$
such that, for $i \geq i_0$, $T_2(M_i)=T_2(M_{i+1})$ and thus, by
lemma \ref{app1'}, $M_i \trs^+ M_{i+1}$.  This contradicts lemma
\ref{app2}. \hfill $\square$

\medskip

\noindent {\bf Remark.}

Assume that, in theorem \ref{app}, $\ep=[y_1. Q_1, y_2.Q_2]$. If
$\ov{V}$ were not empty, the proof of lemma \ref{app1} would not
work because a redex could be created by the transformation $T_1$.
Here is an example:  let $M=(P \; \bo{\ep} \; V)$ be correct and
assume $T_1(M)=(P \; [y_1. Q_1, y_2.Q_2] \; V) \tr N=(P \; [y_1.
(Q_1 \; V), y_2.(Q_2 \; V) ])$.  There is no way to find $M'$ such
that $M \tr M'$ and $T_1(M')=N$ because $(P \; \bo{\ep} \; V)$ is
not a redex.

We do not know wether theorem  \ref{app} remains true if the
sequence $\ep \; \ov{V}$ is not nice: to prove it, $(P \; \bo{\ep}
\; V)$ should then be considered as a redex but $\tst$ becomes
much more complicated. Since  $\ep \; \ov{V}$ is  nice, it is
simpler to add a new condition in the definition \ref{cor} of
correctness to ensure that this situation (of the creation of a
redex by  the transformation $T_1$) does not appear. This
condition is the following:

\begin{center}

{\it 3.} $M$ is good wrt the set of all its sub-terms of the form
   $(U \; \bo{\ep})$.
   \end{center}

\noindent   where, if $E$ is a subset of $\tst$,  $M$ is good wrt
to $E$ is defined by:  $M \in E$ or $M= \m a \; N$ and for each
occurrence of $(a \; N)$ in $M$, $N$ is good wrt to $E$ or $M=(N
\; [x_1. N_1, x_2.N_2])$ and $N_1, N_2$ are good wrt to $E$.

This condition  implies in particular that, in a correct term,
there is no sub-term of the form $(N \; \bo{\ep} U)$ and thus that
lemma \ref{app1} remains valid. It is not difficult to check that
the other lemmas  remain  also  valid.

\end{document}